\documentclass[a4paper]{amsart}
\usepackage{graphicx}
\usepackage{amsmath,amssymb}
\usepackage{amsthm}
\usepackage{here}
\usepackage{comment}
\usepackage{amsaddr}
\usepackage{fancyhdr}
\newtheorem{theorem}{Theorem}[section]

\newtheorem{corollary}[theorem]{Corollary}
\newtheorem{lemma}[theorem]{Lemma}
\newtheorem{proposition}[theorem]{Proposition}
\theoremstyle{definition}
\newtheorem{definition}[theorem]{Definition}
\newtheorem{example}[theorem]{Example}
\theoremstyle{remark}

\newenvironment{psmatrix}{\left(\begin{smallmatrix}}{\end{smallmatrix}\right)}

\newcommand{\BB}{\mathbb{B}}
\newcommand{\CC}{\mathbb{C}}
\newcommand{\DD}{\mathbb{D}}
\newcommand{\FF}{\mathbb{F}}
\newcommand{\RR}{\mathbb{R}}
\newcommand{\TT}{\mathbb{T}}
\newcommand{\ZZ}{\mathbb{Z}}
\newcommand{\EE}{\mathbb{E}}
\newcommand{\LL}{\mathbb{L}}

\newcommand{\hG}{G}
\newcommand{\bG}{\overline{G}}
\newcommand{\bX}{Y}

\renewcommand{\o}{\bm{0}}

\usepackage[whole]{bxcjkjatype}
\usepackage{txfonts}

\numberwithin{equation}{section}
\usepackage{bm}

\usepackage{color}
\definecolor{darkgreen}{cmyk}{1,0,1,.2}
\definecolor{darkorchid}{rgb}{1.0, 0.5, 0}
\definecolor{persimmon}{rgb}{0.93, 0.35, 0.0}

\newdimen\theight
\def\TeXref#1{%
             \leavevmode\vadjust{\setbox0=\hbox{{\tt
                     \quad\quad  {\small \textrm #1}}}%
             \theight=\ht0
             \advance\theight by \lineskip
             \kern -\theight \vbox to
             \theight{\rightline{\rlap{\box0}}%
             \vss}%
             }}%

\begin{document}
\title{Generalized Bishop frames of regular time-like curves in 4-dimensional Lorentz space $\mathbb{L}^{4}$}
\begin{abstract}
    We introduced generalized Bishop frames on curves in 4-dimensional Euclidean space $\mathbb{E}^{4}$, which are orthonormal frames such that the derivatives of the vectors of the frames along the curve can be expressed, via a certain matrix, as a linear combination of the vectors of the frame.
    In relation to that, we study generalized Bishop frames of regular time-like  curves. In a previous work, we showed that there is a hierarchy among different types of generalized Bishop frames for regular curves in the Euclidean space. Building upon this study, we further investigate it in the 4-dimensional Lorentz space $\mathbb{L}^4$. There are four types of generalized Bishop frames of regular time-like curves in $\mathbb{L}^{4}$ up to the change of the order of vectors fixing the first one which is the tangent vector. Unlike other types of curves, such as light-like and space-like ones, the time-like curve can be investigated in a manner analogous to the Euclidean case. We find that a hierarchy of frames exists, similar to that in the Euclidean setting. Based on this hierarchy, we propose a new classification of curves.
 \end{abstract}
\author{Subaru Nomoto$\,{}^{*}$}
\thanks{${}^{*}$ Corresponding author.}
\address{Department of Mathematical Sciences, Colleges of Science and Engineering, Ritsumeikan University, Nojihigashi 1-1-1, Kusatsu, Shiga, 525-8577, Japan}
\email{snomoto@fc.ritsumei.ac.jp}


\date{}


  \maketitle 

\section{Introduction}
For the study of space curves and surfaces, various types of frames play important roles and provide significant information about the objects. For example, the Frenet frame yields curvature and torsion, while the Darboux frame gives geodesic curvature, geodesic torsion, and normal curvature. Curvature and torsion form a complete set of invariants of space curves via the Frenet frame, which is constructed by an algebraic method. While other frames can be constructed analytically, one such frame was introduced by L. R. Bishop [1], and it has certain advantages over the Frenet frame. In particular, every $C^{2}$ regular curve admits a Bishop frame, as shown in Bishop’s theorem \cite{b}, whereas some regular curves do not admit a Frenet frame. Further applications of Bishop frames in differential geometry can be found in \cite{BG,dSdS,KT,MR}, and they have also been applied in other fields, including computer graphics and engineering (see \cite{FL,W}). Ishikawa \cite{GI} studied frontals using a frame that lies between the Bishop frame and the Frenet frame. Related to our generalized Bishop frame, Kazan \cite{KA} investigated canal hypersurfaces. In this article, we continue to study relations among frames of curves obtained by generalizing the Frenet and Bishop frames. Throughout this article we consider smooth curves. Our motivation for this study was inspired by Bishop's considerations. Let us review the idea of Bishop to introduce Bishop frames. Let $\EE^{3}$ be 3-dimensional Euclidean space. For a regular curve $\gamma : I \to \EE^{3}$ defined on an open interval $I$ with an arc-length parametrization, a Bishop frame is a frame of $\gamma^{*}T\mathbb{E}^3$ 
of the form $\{\mathbb{T}, \mathbb{B}_{1}, \mathbb{B}_{2}\}$, where $\mathbb{T}$ is the tangent vector and satisfies the following differential equation: 
\[\begin{pmatrix}
\mathbb{T}^{\prime}\\
\mathbb{B}_{1}^{\prime}\\
\mathbb{B}_{2}^{\prime}
\end{pmatrix}
=\begin{pmatrix}
0&b_{1}&b_{2}\\
-b_{1}&0&0\\
-b_{2}&0&0
\end{pmatrix}
\begin{pmatrix}
\mathbb{T}\\
\mathbb{B}_{1}\\
\mathbb{B}_{2}
\end{pmatrix}
.
\]
Given a frame $\mathbb{F} = \begin{pmatrix}
\mathbb{F}_0 \\
\mathbb{F}_1 \\
\mathbb{F}_2
\end{pmatrix}$, we will call the one-parameter family of matrices $X$ such that $\mathbb{F}' = X \mathbb{F}$ the coefficient matrix of $\mathbb{F}$.
Bishop came up with the frame with the coefficient matrix of this form as follows.
\[
\begin{array}{c c c}
\left(\begin{matrix}
0 & x_{1} & 0 \\
-x_{1} & 0 & x_{2} \\
0 & -x_{2} & 0
\end{matrix}\right), 
& \hspace{40pt} 
\left(\begin{matrix}
0 & x_{1} & x_{2} \\
-x_{1} & 0 & 0 \\
-x_{2} & 0 & 0
\end{matrix}\right),
& \hspace{40pt} 
\left(\begin{matrix}
0 & 0 & x_{1} \\
0 & 0 & x_{2} \\
-x_{1} & -x_{2} & 0
\end{matrix}\right). \\
(1) & \hspace{40pt} (2) & \hspace{40pt} (3)
\end{array}
\]

A frame whose coefficient matrix is of type (2) is precisely the Bishop frame. A frame whose coefficient matrix has the form (1) is closely related to the Frenet frame; it becomes the Frenet frame if $x_1$ is positive.
Bishop observed that if a frame $\{\TT, \ZZ_1, \ZZ_2\}$ has a coefficient matrix of the form (3), then the frame $\{\TT, \ZZ_{2}, \ZZ_{1}\}$ has a coefficient matrix of the form (1). Hence, in 3-dimensional Euclidean space $\EE^{3}$, there are essentially two such types of frames for curves: the Frenet frame and Bishop frame. In the case of time-like curves in 3-dimensional Lorentz space, the coefficient matrices are as follows.
\[
\begin{array}{c c c}
\left(\begin{matrix}
0 & x_{1} & 0 \\
x_{1} & 0 & x_{2} \\
0 & -x_{2} & 0
\end{matrix}\right), 
& \hspace{40pt} 
\left(\begin{matrix}
0 & x_{1} & x_{2} \\
x_{1} & 0 & 0 \\
x_{2} & 0 & 0
\end{matrix}\right),
& \hspace{40pt} 
\left(\begin{matrix}
0 & 0 & x_{1} \\
0 & 0 & x_{2} \\
x_{1} & -x_{2} & 0
\end{matrix}\right). \\
(1) & \hspace{40pt} (2) & \hspace{40pt} (3)
\end{array}
\]
The difference from the Euclidean case lies in the signs of entries of the coefficient matrix. It is also equivalent to (1) and (3) by switching the second and third normal vectors. Therefore, in 3-dimensional Lorentzian space as well, the Frenet and Bishop frames are essentially the only types of such frames. We extend this idea to time-like curves in the 4-dimensional Lorentz space $\LL^{4}$.
Throughout this article, unless otherwise specified, every regular time-like curve is assumed to be parameterized by its proper time \(s\), where
\[
s(u)=\int_{u_0}^{u} \sqrt{-\bigl\langle \tfrac{d\gamma}{dt}(t), \tfrac{d\gamma}{dt}(t) \bigr\rangle}\, dt.
\]
 We consider an orthonormal frame of a curve $\gamma : I \to \LL^{4}$ as a matrix valued function $\ZZ : I \to O(1,3)$ whose row vectors form the frame. For a frame of a regular time-like curve, let X be the matrix-valued function such that $\ZZ'=X\ZZ$.
Let $\gamma : I \to \LL^{4}$ be a regular time-like curve. 
An orthonormal frame $\{\TT, \ZZ_1, \ZZ_2, \ZZ_3\}$ of $\gamma$ is called a Frenet frame if its coefficient matrix $X$ takes the first form below, 
and it is called a Bishop frame if $X$ takes the second form:
\[
\begin{array}{c}
\begin{pmatrix}
\phantom{-}0&\phantom{-}f_{1}&\phantom{-}0&\phantom{-}0\\
f_{1}&\phantom{-}0&\phantom{-}f_{2}&\phantom{-}0\\
\phantom{-}0&-f_{2}&\phantom{-}0&\phantom{-}f_{3}\\
\phantom{-}0&\phantom{-}0&-f_{3}&\phantom{-}0
\end{pmatrix}\\[6pt]
\text{the coefficient matrix of Frenet frame}
\end{array}
\qquad\text{and}\qquad
\begin{array}{c}
\begin{pmatrix}
\phantom{-}0&\phantom{-}b_{1}&\phantom{-}b_{2}&\phantom{-}b_{3}\\
b_{1}&\phantom{-}0&\phantom{-}0&\phantom{-}0\\
b_{2}&\phantom{-}0&\phantom{-}0&\phantom{-}0\\
b_{3}&\phantom{-}0&\phantom{-}0&\phantom{-}0
\end{pmatrix}\\[6pt]
\text{the coefficient matrix of Bishop frame}
\end{array}
\]
respectively, where $b_1, b_2, b_3, f_3$ are functions on $I$ and $f_1, f_2$ are positive functions on $I$. Generalizing the situation for space curves, we now give the following definition:
\begin{definition}
An orthonormal frame of a regular time-like curve in $\LL^{4}$ is called a generalized Bishop frame if its coefficient matrix $(a_{ij})$ has at most three nonzero entries in the strictly upper triangular part. (i.e.\ among the entries $a_{ij}$ with $\, i< j,$ at most three are nonzero.)
\end{definition}
Aside from degenerate cases that contain a zero column vector, 
there are 16 possible types of such frame (see Section \ref{sec:pre}). Moreover, these 16 frames split into four equivalence classes under permutations of the vectors that fix the first one (the tangent vector). Concretely, if a frame's coefficient matrix takes the respective shape for some functions $x_1,x_2, x_3$ up to reordering the last three vectors, we call it a \emph{generalized Bishop frame} of type B, C, D or F for time-like curves, which has the following forms of coefficient matrices up to permutations of the vectors of the frame:
\begin{align*}
\begin{pmatrix}
\phantom{-}0&\phantom{-}x_{1}&\phantom{-}x_{2}&\phantom{-}x_{3}\\
x_{1}&\phantom{-}0&\phantom{-}0&\phantom{-}0\\
x_{2}&\phantom{-}0&\phantom{-}0&\phantom{-}0\\
x_{3}&\phantom{-}0&\phantom{-}0&\phantom{-}0
\end{pmatrix}
, && 
\begin{pmatrix}
\phantom{-}0&\phantom{-}x_{1}&\phantom{-}x_{2}&\phantom{-}0\\
x_{1}&\phantom{-}0&\phantom{-}0&\phantom{-}x_{3}\\
x_{2}&\phantom{-}0&\phantom{-}0&\phantom{-}0\\
\phantom{-}0&-x_{3}&\phantom{-}0&\phantom{-}0
\end{pmatrix}, \\
\text{Type B} \hspace{40pt}&& \text{Type C}\hspace{40pt} 
\end{align*}
\begin{align*}
\begin{pmatrix}
\phantom{-}0&\phantom{-}x_{1}&\phantom{-}0&\phantom{-}0\\
x_{1}&\phantom{-}0&\phantom{-}x_{2}&\phantom{-}x_{3}\\
\phantom{-}0&-x_{2}&\phantom{-}0&\phantom{-}0\\
\phantom{-}0&-x_{3}&\phantom{-}0&\phantom{-}0
\end{pmatrix}, &&
\begin{pmatrix}
\phantom{-}0&\phantom{-}x_{1}&\phantom{-}0&\phantom{-}0\\
x_{1}&\phantom{-}0&\phantom{-}x_{2}&\phantom{-}0\\
\phantom{-}0&-x_{2}&\phantom{-}0&\phantom{-}x_{3}\\
\phantom{-}0&\phantom{-}0&-x_{3}&\phantom{-}0
\end{pmatrix}. \\
\text{Type D}\hspace{40pt} && \text{Type F}\hspace{40pt} 
\end{align*}
In this article, a generalized Bishop frame of type B, C, D or F may be referred to simply as a frame of type B, C, D, or F, respectively. A frame of type B will  also be called a Bishop frame.
\,A frame of type F is closely related to the Frenet frame ; specifically, both $x_1$ and $x_2$ are positive.
A frame of type B corresponds to the original Bishop frame in Euclidean space. Analogous to Bishop's theorem in the Euclidean case, one can show that every regular time-like curve admits a frame of type B. Ishikawa \cite{GI} used a frame of type D along regular curves on the Euclidean space to study frontals in the Euclidean space. 

This paper present a hierarchy among four types of frames associated with time-like curves in 4-dimensional Lorentz space $\LL^{4}$. The results are parallel to those in \cite{NN}, but we employed different methods to prove certain lemmas.
\begin{theorem}\label{thm:hier}
Let $\gamma$ be a regular time-like curve in $\LL^4$.
\begin{enumerate}
    \item If $\gamma$ admits a generalized Bishop frame of type F, then $\gamma$ also admits a generalized Bishop frame of type D. 
    \item If $\gamma$ admits a generalized Bishop frame of type D, then $\gamma$ also admits a generalized Bishop frame of type C.  
\end{enumerate}
\end{theorem}
This hierarchy is proper in the following sense. 
\begin{theorem}\label{thm:noCD}
\begin{enumerate}
    \item There exists a regular time-like curve that admits a generalized Bishop frame of type D but does not admit one of type F. 
    \item There exists a regular time-like curve that admits a generalized Bishop frame of type C but does not admit one of type D. 
    \item There exists a regular time-like curve that does not admit a generalized Bishop frame of type C.
\end{enumerate}
\end{theorem}

Let us consider a frame for the following type of curve.
\begin{definition}[{\cite{NN}}]
A time-like curve is called \emph{$2$-regular} if both its tangent vector and its derivative are nowhere zero. 
\end{definition}

From Theorem \ref{thm:hier}, we can derive the following corollary.
\begin{corollary}\label{thm:2reg}
If a time-like $2$-regular curve $\gamma$ in $\LL^{4}$ admits a frame of type F, then it also admits generalized Bishop frames of all the other types B, C and D. In particular, if $\gamma$ admits the Frenet frame, then it admits frames of all other types.
\end{corollary}
The next result shows that this hierarchy for time-like $2$-regular curves is also proper.

\begin{theorem}\label{thm:noF}
There is a time-like $2$-regular curve in $\LL^{4}$ that does not admit a generalized Bishop frame of type F.
\end{theorem}
In analogy with Bishop's result, every $C^{2}$ regular time-like curve admits a frame of type B. We will prove the Lorentz version of Bishop's theorem. Note that the part of Corollary \ref{thm:2reg} confirming the existence of a generalized Bishop frame of type B is essentially Bishop's theorem in the Lorentz setting.

Using the Frenet frame for a curve makes it possible to compute specific geometric invariants of the curve, such as its curvature and torsion.
\ However, some curves do not admit the Frenet frame. By taking advantage of that fact, we can view this situation as a means of classifying curves.
For space-like curve and light-like curve, we can not always define unit vectors for frames, so the problem is more complicate.
\section{Basic properties of frames}\label{sec:pre}
In $4$-dimensional Lorentz space $\LL^{4}$, we can consider three types of vectors. Let us review the definitions of three types of vectors.
\begin{definition}[{\cite{bo}}]
Let $\bm{v}$ be a vector in $\LL^{4}$.
\begin{enumerate}
\item if \(\langle \bm{v}, \bm{v}\rangle<0\), we call \(\bm{v}\) a time-like vector.
\item if \(\bm{v}=\bm{0}\) or \(\langle \bm{v}, \bm{v}\rangle>0\), we call \(\bm{v}\) a space-like vector.
\item if \(\bm{v}\neq\bm{0}\) and \(\langle \bm{v}, \bm{v}\rangle=0\), we call \(\bm{v}\) a light-like vector.
\end{enumerate}    
\end{definition}
According to the Lorentzian metric, the tangent vector of a regular curve at each point is 
either time-like, space-like, or null. First, we note that the norm of a time-like or space-like vector $\bm{v}$ in $\LL^{4}$ is defined by $\|\bm{v}\|=\sqrt{|\langle\bm{v},\bm{v}\rangle|}$. Furthermore even there exists three types of vectors in $\LL^{4}$, time-like, space-like and null vector, we can exclude the possibility a curve having normal vectors as time-like or null in the moving frame with the following lemma\,\cite{YM}.
\begin{lemma}\label{lorentz}
    For a non zero vector $\bm{v}\in\RR^{n+1}_{1}$, the orthogonal complement $\bm{v}^{\perp} \coloneqq\{\bm{w}\in\RR^{n+1}_{1}\mid \langle\bm{v}, \bm{w}\rangle=0\}$ satisfies the followings.
    \begin{enumerate}
    \item $\bm{v}^{\perp}$ is a n-dimensional subspace of $\RR^{n+1}_{1}$.
    \item If $\bm{v}$ is a time-like vector, $\bm{v}^{\perp}$ is a space-like subspace of $\RR^{n+1}$, which means $\bm{v}^{\perp}$ has bases with space-like vectors only.
    \item If $\bm{v}$ is a space-like vector, the restriction of the inner product to $\bm{v}^{\perp}$ has the signature $(n-1, 1)$.
    \end{enumerate}
    \end{lemma}

By restricting our attention to time-like curves, we can investigate a case nearly identical to that for curves in Euclidean space. Thanks to this lemma, we can classify the frames of time-like curves in a way similar to the case of regular curves in $\mathbb{E}^4$. In $\LL^{4}$, there are in fact 16 possible generalized Bishop frames for a regular time-like curve (excluding those whose coefficient matrix contains a zero column). These frames can be classified into four types under the action of the symmetric group $\mathfrak{S}_{3}$, which permutes the second, third and fourth vectors of the frame. Symbolically, one may represent them using a matrix of the form:\\
\begin{minipage}[b]{0.48\columnwidth}
 \begin{align*}
\begin{psmatrix}
\phantom{-}0&\phantom{-}f_{1}&\phantom{-}0&\phantom{-}0\\
f_{1}&\phantom{-}0&\phantom{-}f_{2}&\phantom{-}0\\
\phantom{-}0&-f_{2}&\phantom{-}0&\phantom{-}f_{3}\\
\phantom{-}0&\phantom{-}0&-f_{3}&\phantom{-}0
\end{psmatrix}
\begin{psmatrix}
\phantom{-}0&\phantom{-}0&\phantom{-}0&\phantom{-}f_{1}\\
\phantom{-}0&\phantom{-}0&\phantom{-}f_{2}&\phantom{-}f_{3}\\
\phantom{-}0&-f_{2}&\phantom{-}0&\phantom{-}0\\
f_{1}&-f_{3}&\phantom{-}0&\phantom{-}0
\end{psmatrix}
\begin{psmatrix}
\phantom{-}0&\phantom{-}0&\phantom{-}0&\phantom{-}f_{1}\\
\phantom{-}0&\phantom{-}0&\phantom{-}f_{2}&\phantom{-}0\\
\phantom{-}0&-f_{2}&\phantom{-}0&\phantom{-}f_{3}\\
f_{1}&\phantom{-}0&-f_{3}&\phantom{-}0
\end{psmatrix}\\
\begin{psmatrix}
\phantom{-}0&\phantom{-}f_{1}&\phantom{-}0&\phantom{-}0\\
f_{1}&\phantom{-}0&\phantom{-}0&\phantom{-}f_{2}\\
\phantom{-}0&\phantom{-}0&\phantom{-}0&\phantom{-}f_{3}\\
\phantom{-}0&-f_{2}&-f_{3}&\phantom{-}0
\end{psmatrix}
\begin{psmatrix}
\phantom{-}0&\phantom{-}0&\phantom{-}f_{1}&\phantom{-}0\\
\phantom{-}0&\phantom{-}0&\phantom{-}f_{2}&\phantom{-}f_{3}\\
f_{1}&-f_{2}&\phantom{-}0&\phantom{-}0\\
\phantom{-}0&-f_{3}&\phantom{-}0&\phantom{-}0
\end{psmatrix}
\begin{psmatrix}
\phantom{-}0&\phantom{-}0&\phantom{-}f_{1}&\phantom{-}0\\
\phantom{-}0&\phantom{-}0&\phantom{-}0&\phantom{-}f_{2}\\
f_{1}&\phantom{-}0&\phantom{-}0&\phantom{-}f_{3}\\
\phantom{-}0&-f_{2}&-f_{3}&\phantom{-}0
\end{psmatrix}
\end{align*}
\begin{center}type F\end{center}
\end{minipage}
\begin{minipage}[b]{0.48\columnwidth}
\begin{align*}
\begin{psmatrix}
\phantom{-}0&\phantom{-}c_{1}&\phantom{-}c_{2}&\phantom{-}0\\
c_{1}&\phantom{-}0&\phantom{-}0&\phantom{-}c_{3}\\
c_{2}&\phantom{-}0&\phantom{-}0&\phantom{-}0\\
\phantom{-}0&-c_{3}&\phantom{-}0&\phantom{-}0
\end{psmatrix}
\begin{psmatrix}
\phantom{-}0&\phantom{-}c_{1}&\phantom{-}c_{2}&\phantom{-}0\\
c_{1}&\phantom{-}0&\phantom{-}0&\phantom{-}0\\
c_{2}&\phantom{-}0&\phantom{-}0&\phantom{-}c_{3}\\
\phantom{-}0&\phantom{-}0&-c_{3}&\phantom{-}0
\end{psmatrix}
\begin{psmatrix}
\phantom{-}0&\phantom{-}c_{1}&\phantom{-}0&\phantom{-}c_{2}\\
c_{1}&\phantom{-}0&\phantom{-}0&\phantom{-}0\\
\phantom{-}0&\phantom{-}0&\phantom{-}0&\phantom{-}c_{3}\\
c_{2}&\phantom{-}0&-c_{3}&\phantom{-}0
\end{psmatrix}\\
\begin{psmatrix}
\phantom{-}0&\phantom{-}c_{1}&\phantom{-}0&\phantom{-}c_{2}\\
c_{1}&\phantom{-}0&\phantom{-}c_{3}&\phantom{-}0\\
\phantom{-}0&-c_{3}&\phantom{-}0&\phantom{-}0\\
c_{2}&\phantom{-}0&\phantom{-}0&\phantom{-}0
\end{psmatrix}
\begin{psmatrix}
\phantom{-}0&\phantom{-}0&\phantom{-}c_{1}&\phantom{-}c_{2}\\
\phantom{-}0&\phantom{-}0&\phantom{-}c_{3}&\phantom{-}0\\
c_{1}&-c_{3}&\phantom{-}0&\phantom{-}0\\
c_{2}&\phantom{-}0&\phantom{-}0&\phantom{-}0
\end{psmatrix}\begin{psmatrix}
\phantom{-}0&\phantom{-}0&\phantom{-}c_{1}&\phantom{-}c_{2}\\
\phantom{-}0&\phantom{-}0&\phantom{-}0&\phantom{-}c_{3}\\
c_{1}&\phantom{-}0&\phantom{-}0&\phantom{-}0\\
c_{2}&-c_{3}&\phantom{-}0&\phantom{-}0
\end{psmatrix}
\end{align*}
\begin{center}type C\end{center}
\end{minipage}
\begin{minipage}[b]{0.48\columnwidth}
\[\begin{psmatrix}
\phantom{-}0&\phantom{-}d_{1}&\phantom{-}0&\phantom{-}0\\
d_{1}&\phantom{-}0&\phantom{-}d_{2}&\phantom{-}d_{3}\\
\phantom{-}0&-d_{2}&\phantom{-}0&\phantom{-}0\\
\phantom{-}0&-d_{3}&\phantom{-}0&\phantom{-}0
\end{psmatrix}
\begin{psmatrix}
\phantom{-}0&\phantom{-}0&\phantom{-}d_{1}&\phantom{-}0\\
\phantom{-}0&\phantom{-}0&\phantom{-}d_{2}&\phantom{-}0\\
d_{1}&-d_{2}&\phantom{-}0&\phantom{-}d_{3}\\
\phantom{-}0&\phantom{-}0&-d_{3}&\phantom{-}0
\end{psmatrix}
\begin{psmatrix}
\phantom{-}0&\phantom{-}0&\phantom{-}0&\phantom{-}d_{1}\\
\phantom{-}0&\phantom{-}0&\phantom{-}0&\phantom{-}d_{2}\\
\phantom{-}0&\phantom{-}0&\phantom{-}0&\phantom{-}d_{3}\\
d_{1}&-d_{2}&-d_{3}&\phantom{-}0\\
\end{psmatrix}
\]
\begin{center}type D\end{center}
\end{minipage}
\begin{minipage}[b]{0.48\columnwidth}
\[\begin{psmatrix}
\phantom{-}0&\phantom{-}b_{1}&\phantom{-}b_{2}&\phantom{-}b_{3}\\
b_{1}&\phantom{-}0&\phantom{-}0&\phantom{-}0\\
b_{2}&\phantom{-}0&\phantom{-}0&\phantom{-}0\\
b_{3}&\phantom{-}0&\phantom{-}0&\phantom{-}0
\end{psmatrix}\]
\begin{center}type B\end{center}
\end{minipage}\\
In 4-dimensional Euclidean space, a skew-symmetric matrix appears as a coefficient matrix. On the other hand, in 4-dimensional Lorentz space, coefficient matrices are as above.
A frame of type B has a symmetric coefficient matrix. 
The proofs of hierarchy of frames in Lorentz space are almost the same as the proofs in Euclidean space but the characterization of curves which admit a frame of type D was different and it was the most important theorem for leading other theorem. First we will see Bishop's theorem in the Lorentz space following \cite{MR}.
\begin{lemma}\label{lemma:0}
Every $C^{2}$ regular time-like curve in $\LL^{4}$ admits a frame of type B. 
\end{lemma}
\begin{proof}
Let $\gamma$ be a $C^{2}$ regular time-like curve in $\mathbb{L}^4$. By Lemma 2.1, there exist space-like vectors at a point $\gamma(s_0)$, and we can construct a basis $\{\mathbb{T}^0, \mathbb{B}_1^0, \mathbb{B}_2^0, \mathbb{B}_3^0\}$ at $\gamma(s_0)$ using Gram–Schmidt orthonormalization in Lorentz space with using suitable vectors and we have the following conditions. \[\langle\mathbb{T}^0, \mathbb{T}^0\rangle=-1,\quad \langle\mathbb{B}^{0}_i, \mathbb{B}^{0}_{j}\rangle=\delta_{ij} ,\quad \langle\mathbb{T}^0, \mathbb{B}^{0}_{i}\rangle=0.\] Next we consider the following differential equations for a vector-valued function of 
class $C^{1}$ with respect to s:
\[
\mathbb{B}_1^{'}(s) = \langle\mathbb{T}^{\prime}(s), \mathbb{B}_{1}(s)\rangle\mathbb{T}(s), \quad \mathbb{B}_2^{'}(s) = \langle\mathbb{T}^{\prime}(s), \mathbb{B}_{2}(s)\rangle\mathbb{T}(s), \quad \mathbb{B}_3^{'}(s) = \langle\mathbb{T}^{\prime}(s), \mathbb{B}_{3}(s)\rangle\mathbb{T}(s)
\]
with the initial conditions:
\[
\mathbb{T}(s_0) = \mathbb{T}^0, \quad \mathbb{B}_1(s_0) = \mathbb{B}_1^0, \quad \mathbb{B}_2(s_0) = \mathbb{B}_2^0, \quad \mathbb{B}_3(s_0) = \mathbb{B}_3^0
.\]
By existence and uniqueness of solutions for ODEs, we have its solutions locally.
For investigating $B_{1}(s), B_{2}(s)$ and $B_{3}(s)$ are constructing an orthonormal moving frame, 
first we differentiate $\langle\mathbb{T}(s), \mathbb{T}
(s)\rangle=-1$ with s:
\[
(\langle\mathbb{T}(s), \mathbb{T}
(s)\rangle)^{\prime}=2\langle\mathbb{T}^{\prime}(s), \mathbb{T}(s)\rangle=0    
.\]
Therefore $\mathbb{T}^{\prime}(s)$ and $\mathbb{T}(s)$ are orthogonal to each other.
We differentiate $\langle \mathbb{T}(s),\mathbb{B}_1(s)\rangle$ with $s$:
\begin{align*}
\left( \langle \mathbb{T}(s), \mathbb{B}_1(s) \rangle \right)' 
= &\langle\mathbb{T}^{\prime}(s), \mathbb{B}_{1}(s)\rangle+\langle\mathbb{T}(s), \mathbb{B}_{1}^{\prime}(s)\rangle\\
=&\langle\mathbb{T}^{\prime}(s), \mathbb{B}_{1}(s)\rangle+\langle\mathbb{T}(s), \langle\mathbb{T}^{\prime}(s), \mathbb{B}_{1}(s)\rangle\mathbb{T}(s)\rangle\\
=&\langle\mathbb{T}^{\prime}(s), \mathbb{B}_{1}(s)\rangle- \langle\mathbb{T}^{\prime}(s), \mathbb{B}_{1}(s)\rangle\\
=&0
.\end{align*}
Hence $\langle\mathbb{T}(s), \mathbb{B}_{1}(s)\rangle$ is identically zero and
 $\mathbb{T}(s)$ and $\mathbb{B}_1(s)$ remain orthogonal. Similarly, $\mathbb{B}_2(s)$ and $\mathbb{B}_3(s)$ are also orthogonal to $\mathbb{T}(s)$.
Next we differentiate $\langle \mathbb{B}_{1}(s),\mathbb{B}_1(s)\rangle$ with $s$:
\begin{align*}
\left( \langle \mathbb{B}_{1}(s), \mathbb{B}_1(s) \rangle \right)' 
= &2\langle\mathbb{B}^{\prime}_{1}(s), \mathbb{B}_{1}(s)\rangle\\
=&2\langle\langle\mathbb{T}^{\prime}(s), \mathbb{B}_{1}(s)\rangle\mathbb{T}(s), \mathbb{B}_{1}(s)\rangle\\
=&0.
\end{align*}
Therefore $\langle\mathbb{B}_{1}(s), \BB_{1}(s)\rangle$ is identically one. $\langle\mathbb{B}_{2}(s),\mathbb{B}_{2}(s)\rangle$ and $\langle\mathbb{B}_{3}(s),\mathbb{B}_{3}(s)\rangle$ are also identically one. 
Next we differentiate $\langle \mathbb{B}_{i}(s),\mathbb{B}_{j}(s)\rangle\,(i\neq j)$ with $s$:
\begin{align*}
\left( \langle \mathbb{B}_{i}(s), \mathbb{B}_{j}(s) \rangle \right)' 
= &\langle\mathbb{B}^{\prime}_{i}(s), \mathbb{B}_{j}(s)\rangle+\langle\mathbb{B}_{i}(s), \mathbb{B}^{\prime}_{j}(s)\rangle\\
=&\langle\langle\mathbb{T}^{\prime}(s), \mathbb{B}_{i}(s)\rangle\mathbb{T}(s), \mathbb{B}_{j}(s)\rangle+\langle\mathbb{B}_{i}(s), \langle\mathbb{T}^{\prime}(s), \mathbb{B}_{j}(s)\rangle\mathbb{T}(s)\rangle\\
=&0.
\end{align*}
Therefore $\mathbb{B}_{i}(s)$ and $\mathbb{B}_{j}(s)$ are orthogonal to each other.
Thus $\{\mathbb{T}(s), \mathbb{B}_{1}(s), \mathbb{B}_{2}(s), \mathbb{B}_{3}(s)\}$ is an orthonormal moving frame. We will express $\mathbb{T}^{\prime}(s)$ with the moving frame and we express it as $\mathbb{T}^{\prime}(s)=\alpha_{1}(s)\mathbb{B}_{1}(s)+\alpha_{2}(s)\mathbb{B}_{2}(s)+\alpha_{3}(s)\mathbb{B}_{3}(s)$. Since $\langle\TT(s), \BB_{i}(s)\rangle^{\prime}=0\, (i=1,2,3)$, we find $\alpha_{1}(s)=\langle\mathbb{T}^{\prime}(s), \mathbb{B}_{1}(s)\rangle, \alpha_{2}(s)=\langle\mathbb{T}^{\prime}(s), \mathbb{B}_{2}(s)\rangle$ and $\alpha_{3}(s)=\langle\mathbb{T}^{\prime}(s), \mathbb{B}_{3}(s)\rangle$. Here we put $b_{1}(s)\coloneqq \langle\mathbb{T}^{\prime}(s), \mathbb{B}_{1}(s)\rangle, b_{2}(s)\coloneqq \langle\mathbb{T}^{\prime}(s), \mathbb{B}_{2}(s)\rangle$ and $b_{3}(s)\coloneqq \langle\mathbb{T}^{\prime}(s), \mathbb{B}_{3}(s)\rangle$. Thus, we have constructed a frame of type B.
\end{proof}
\begin{lemma}\label{lemma:1}
Every $C^{2}$ regular space-like curve in $\LL^{4}$ admits a frame of type B whose first normal vector $\mathbb{B}_{1}(s)$ is time-like. 
\end{lemma}
\begin{proof}

Let $\gamma$ be a $C^{\infty}$ space-like curve in $\mathbb{L}^4$. By Lemma 2.1, there exist space-like vectors and a time-like vector at a point $\gamma(s_0)$, and we can construct a basis $\{\mathbb{T}^0, \mathbb{B}_1^0, \mathbb{B}_2^0, \mathbb{B}_3^0\}$ at $\gamma(s_0)$ using Gram–Schmidt orthonormalization in Lorentz space with using suitable vectors. Without loss of generality, we may assume that $\mathbb{B}_{1}^{0}$ is chosen as a time‑like vector. We have the following conditions. 
\begin{align*}
\langle\mathbb{T}^0, \mathbb{T}^0\rangle &=1, &&& \langle\mathbb{B}^{0}_1, \mathbb{B}^{0}_{1}\rangle & =-1, &&& \langle\mathbb{B}^{0}_1,\mathbb{B}^{0}_{i}\rangle & =0\,(i=2,3), 
\end{align*}

\vspace{-20pt}

\begin{align*}
\langle\mathbb{B}^{0}_i, \mathbb{B}^{0}_{j}\rangle &=\delta_{ij}\,(i,j\in\{2,3\}), & \langle\mathbb{T}^0, \mathbb{B}^{0}_{i}\rangle &=0\,(i=1,2,3).
\end{align*}
Next we consider the following differential equations for vector-valued functions of 
class $C^{\infty}$ with respect to s:
\begin{align*}
\mathbb{B}_1^{'}(s) & = -\langle\mathbb{T}^{\prime}(s), \mathbb{B}_{1}(s)\rangle\mathbb{T}(s), &&&  \mathbb{B}_2^{'}(s) &= -\langle\mathbb{T}^{\prime}(s), \mathbb{B}_{2}(s)\rangle\mathbb{T}(s), &&& \mathbb{B}_3^{'}(s) & =-\langle\mathbb{T}^{\prime}(s), \mathbb{B}_{3}(s)\rangle\mathbb{T}(s)
\end{align*}
with the initial conditions:
\[
\mathbb{T}(s_0) = \mathbb{T}^0, \quad \mathbb{B}_1(s_0) = \mathbb{B}_1^0, \quad \mathbb{B}_2(s_0) = \mathbb{B}_2^0, \quad \mathbb{B}_3(s_0) = \mathbb{B}_3^0
\]
We can construct a frame of type B in almost the same way as in the time-like case. 
Therefore, we focus only on the parts that differ in the space-like case. 
Let
\[
\mathbb{T}'(s)=\alpha_{1}(s)\mathbb{B}_1(s)+\alpha_{2}(s)\mathbb{B}_2(s)+\alpha_{3}(s)\mathbb{B}_3(s).
\] 
Since $\langle\TT(s), \BB_{i}(s)\rangle=0\,(i=1,2,3)$, we find $\alpha_{1}(s)=-\langle\mathbb{T}^{\prime}(s), \mathbb{B}_{1}(s)\rangle, \alpha_{2}(s)=\langle\mathbb{T}^{\prime}(s), \mathbb{B}_{2}(s)\rangle$ and $\alpha_{3}(s)=\langle\mathbb{T}^{\prime}(s), \mathbb{B}_{3}(s)\rangle$. Here we put $b_{1}(s)\coloneqq-\langle\mathbb{T}^{\prime}(s), \BB_{1}(s)\rangle, b_{2}(s)\coloneqq\langle\mathbb{T}^{\prime}(s),\mathbb{B}_{2}(s)\rangle$ and $b_{3}(s)\coloneqq\langle\TT^{\prime}(s), \BB_{3}(s)\rangle.$ Thus, we have constructed a frame of type B.  
\end{proof}
If $\gamma$ is a time-like curve, the associated frame $\ZZ$ is in $O(1, 3)$. In \cite{NN}, we were able to switch the tangent vector and normal vector because there is no distinction between vectors with respect to the metric. However, for time-like curves, such a switch cannot be carried out using the original frames. For a proof of  characterization of frame of type D, we will use an existence of a fame of type B on a space-like curve in $\LL^{4}$.
\begin{proposition}\label{proposition:1}
For a regular time-like curve $\gamma$ in $\LL^{4}$, the following are equivalent:
\begin{enumerate}
    \item $\gamma$ admits a frame of type D.
    \item There exists a smooth unit normal space-like vector field $\mathbb{D}_{1}$ and a smooth function $d_1$ such that  $\mathbb{T}^{\prime}=d_{1}\mathbb{D}_{1}$.
\end{enumerate}
\end{proposition}
\begin{proof}
Clearly (1) implies (2). Let us show (1) assuming (2). Let $\gamma$ be a regular time-like curve which satisfies the condition (2). Let $\DD_{1}$ be a unit normal space-like vector field on $\gamma$ such that $\TT^{\prime}=d_{1}\DD_{1}$. Let us consider a regular space-like curve $\delta$ whose tangent vector is $\DD_{1}$. By Lemma \ref{lemma:1}, $\delta$ admits a frame of type B of the form $\{\DD_{1}, \TT, \DD_{2}, \DD_{3}\}$, where $\TT$ is a time-like vector. Then we have
\[\begin{pmatrix}
\DD_{1}^{\prime}\\
\TT^{\prime}\\
\DD_{2}^{\prime}\\
\DD_{3}^{\prime}
\end{pmatrix}=\begin{pmatrix}
0&d_{1}&x_{2}&x_{3}\\
d_{1}&0&0&0\\
-x_{2}&0&0&0\\
-x_{3}&0&0&0
\end{pmatrix}\begin{pmatrix}
\DD_{1}\\
\TT\\
\DD_{2}\\
\DD_{3}
\end{pmatrix}\]
for some functions $d_{1}, x_{2}, x_{3}$. We will consider the frame obtained by switching the first and second vectors and translating the frame along $\gamma$. We have a frame of type D of the following form on $\gamma$
\[\begin{pmatrix}
\TT^{\prime}\\
\DD_{1}^{\prime}\\
\DD_{2}^{\prime}\\
\DD_{3}^{\prime}
\end{pmatrix}=\begin{pmatrix}
0&d_{1}&0&0\\
d_{1}&0&x_{2}&x_{3}\\
0&-x_{2}&0&0\\
0&-x_{3}&0&0
\end{pmatrix}\begin{pmatrix}
\TT\\
\DD_{1}\\
\DD_{2}\\
\DD_{3}
\end{pmatrix}\].
\end{proof}
First, the author considered  we cannot take the similar way of Euclidean case because we need two time-like curve for switching a tangent vector and normal vector. Hiraku Nozawa suggested some advice, we were able to prove Proposition \ref{proposition:1}. By Proposition \ref{proposition:1}, we have a combination of Theorem \ref{thm:hier}-(1) and a part of Corollary \ref{thm:2reg}. As an immediate consequence, the following corollary is obtained.
\begin{corollary}\label{cor:2regD}
Let $\gamma$ be a regular time-like curve in $\LL^4$. If either $\gamma$ is $2$-regular or $\gamma$ admits a frame of type F, then $\gamma$ admits a frame of type D.
\end{corollary}
By Proposition \ref{proposition:1}, we can easily construct an example of a regular time-like curve that does not admit a frame of type D. This proves Theorem \ref{thm:noCD}-(2) and is illustrated by Example \ref{ex:noD}.
\begin{example}\label{ex:noD}
Let $\gamma$ be a regular time-like curve in $\LL^{4}$ defined by
\begin{equation}\label{eq:gamma}
\gamma(t)=
\begin{cases}(t, e^{-\frac{1}{t}},0,0) & t>0\\
(0, 0, 0, 0) & t=0\\
(t,0, e^{\frac{1}{t}},0) & t<0.
\end{cases}
\end{equation}
Note that the parameter of $\gamma$ is not a proper time. As in \cite{NN}, this example is a time-like curve which does not admit a frame of type D but a frame of type C.
\end{example}

We can consider a characterization of curves that admit a frame of type C but it is more involved because we use a frame of type F. It dose not always exist for a given curve.

\begin{proposition}\label{prop:charC}
A regular time-like curve $\gamma$ admits a frame of type C if and only if there exists a unit normal vector field $\ZZ_{0}$ on $\gamma$ such that a space-like curve whose tangent vector is $\ZZ_{0}$ admits a frame of type F of the form $\{\ZZ_{0}, \TT, \ZZ_{1}, \ZZ_{3}\}$, or of the form $\{\ZZ_{0}, \ZZ_{1}, \TT, \ZZ_{3}\}$
whose coefficient matrix is of the form respectively.
\begin{equation*}\label{eq:frenet}
\begin{pmatrix}
\phantom{-}0&\phantom{-}f_{1}&\phantom{-}0&\phantom{-}0\\
\phantom{-}f_{1}&\phantom{-}0&\phantom{-}f_{2}&\phantom{-}0\\
\phantom{-}0&\phantom{-}f_{2}&\phantom{-}0&\phantom{-}f_{3}\\
\phantom{-}0&\phantom{-}0&-f_{3}&\phantom{-}0
\end{pmatrix}\,or\,\begin{pmatrix}
\phantom{-}0&\phantom{-}f_{1}&\phantom{-}0&\phantom{-}0\\
\phantom{-}-f_{1}&\phantom{-}0&\phantom{-}f_{2}&\phantom{-}0\\
\phantom{-}0&\phantom{-}f_{2}&\phantom{-}0&\phantom{-}f_{3}\\
\phantom{-}0&\phantom{-}0&f_{3}&\phantom{-}0
\end{pmatrix}.
\end{equation*}
\end{proposition}

\begin{proof}
It is easy to see that the coefficient matrix of the frame $\{\ZZ_{0}, \TT, \ZZ_{1}, \ZZ_{3}\}$ is of the form left one if and only if we have
\[
\begin{pmatrix}
\mathbb{T}^{\prime}\\
\ZZ_{1}^{\prime}\\
\ZZ_{0}^{\prime}\\
\ZZ_{3}^{\prime}
\end{pmatrix}=\begin{pmatrix}
0&f_{2}&f_{1}&0\\
f_{2}&0&0&f_{3}\\
f_{1}&0&0&0\\
0&-f_{3}&0&0
\end{pmatrix}\begin{pmatrix}
\mathbb{T}\\
\ZZ_{1}\\
\ZZ_{0}\\
\ZZ_{3}
\end{pmatrix}\,.
\]
The latter hold if and only if $\{\TT,\ZZ_{0},\ZZ_{1}\ZZ_{3}\}$ is a frame of type C.
\end{proof}

The next statement is the relation among a frame of type F and the Frenet frame on a time-like curve. The proof is almost the same as in the Euclidean case in \cite{NN}.

\begin{proposition}\label{prop:FFrenet}
If a regular time-like curve $\gamma$ in $\LL^4$ admits the Frenet frame $\{\TT, \FF_1, \FF_2, \FF_3\}$, then every generalized Bishop frame $\{\TT, \ZZ_1, \ZZ_2, \ZZ_3\}$ of type F of $\gamma$ coincides with the Frenet frame up to sign, namely, we have $\ZZ_i = \epsilon_i \FF_i$ for some $\epsilon_i \in \{-1,1\}$ $(i=1,2,3)$. Every entry of the coefficient matrix of the frame $\{\TT, \ZZ_1, \ZZ_2, \ZZ_3\}$ is equal to the corresponding entry of the coefficient matrix of the Frenet frame up to sign. 
\end{proposition}

Using the above proposition, we can show an example of a 2-regular time-like curve which does not admit a frame of type F. In the Euclidean case, we made a curve with the tangent vector field constructed by stereographic projection. This time, we made the curve in an easier way.
\begin{example}\label{ex:noF}
Consider a time-like curve $\gamma$ in $\LL^{4}$ whose tangent vector field is
\begin{equation}
\TT(s)=\begin{cases}
\displaystyle \left(\sqrt{s^{2}e^{-\frac{2}{s}}+s^{2}+1}, \, s e^{\frac{-1}{s}}, \, 0, \, s\right) & s> 0\\
(1, 0, 0, 0) & s=0\\
\displaystyle \left(\sqrt{s^{2}e^{\frac{2}{s}}+s^{2}+1},\, 0, \, s e^{\frac{1}{s}}, \, s\right) & s< 0.\end{cases}
\end{equation}
Differentiating it, we have
\begin{equation}
\TT^{\prime}(s)=\begin{cases}
\displaystyle \left(\frac{e^{-\frac{2}{s}}(1+s+e^{\frac{2}{s}}s)}{\sqrt{1+(1+e^{-\frac{2}{s}})s^{2}}}, \frac{e^{-\frac{1}{s}}(1+s)}{s}, 0, 1\right) & s> 0\\
(0, 0, 0, 1) & s=0\\
\displaystyle \left(\frac{e^{\frac{2}{s}}(-1+s)+s}{\sqrt{1+(1+e^{\frac{2}{s}})s^{2}}}, 0, \frac{e^{\frac{1}{s}}(-1+s)}{s}, 1\right) & s< 0.\end{cases}
\end{equation}
So the curve is 2-regular. By Proposition \ref{proposition:1}, the curve admits a frame of type D and the proof of not admitting a frame of type F is given by the same way as in \cite{NN}. The example \ref{ex:noF} implies Theorem \ref{thm:noF} and Theorem \ref{thm:noCD}-(1)
\end{example}

\section{Transformations between frames}
We now show the remaining parts of the hierarchy of frames by using transformations between them. In \cite{NN}, these results were demonstrated in the Euclidean case, and the same ideas hold for the Lorentz case with minor modifications to certain assumptions. Let $\ZZ_0$ and $\ZZ_1$ be frames of some type, and let $X_0$ and $X_1$ be their corresponding coefficient matrices. For our generalized Bishop frames, we have the relations $\ZZ_0' = X_0 \ZZ_0$ and $\ZZ_1' = X_1 \ZZ_1$. When we differentiate $\hG = \ZZ_1 \ZZ_0^{-1}$, we obtain 
\[
\hG' = \ZZ'_1 \ZZ_0^{-1} + \ZZ_1 (\ZZ_0^{-1})' = X_1 \ZZ_1\ZZ_0^{-1} - \ZZ_1 \ZZ_0^{-1} X_0 = X_{1}{\hG}-{\hG}X_{0}.  
\]
This equation $\hG' = X_{1}{\hG}-{\hG}X_{0}$ of the transformation between frames is the key to the arguments in this section by the following lemma.
\begin{lemma}\label{lem:1}
Let $\gamma : I \to \LL^{4}$ be a regular time-like curve. Let $\ZZ_0$ be a frame of $\gamma$ such that $\ZZ'_0 = X_0 \ZZ_0$, and a function $X_1 : I \to \mathfrak{o}(1,3)$. Consider a frame $\ZZ_1$ of $\gamma$ such that 
\begin{equation}\label{eq:X1}
\ZZ'_1 = X_1 \ZZ_1. 
\end{equation}
Then the transformation $\hG : I \to O (1,3)$ from $\ZZ_0$ to $\ZZ_1$ given by $\hG = \ZZ_1 \ZZ_0^{-1}$ satisfies the differential equation
\begin{equation}
\label{eq:a}{\hG}^{\prime}=X_{1}{\hG}-{\hG}X_{0}.
\end{equation}
Conversely, if a function $\hG : I \to O (1,3)$ satisfies the differential equation \eqref{eq:a}, then the frame $\ZZ_1$ given by $\ZZ_1=\hG \ZZ_0$ satisfies the differential equation \eqref{eq:X1}.
\end{lemma}

After $G$ acting  on the frames, $\TT$ remains $\TT$ in the frames, therefore we have a matrix form $G\in\left\{\begin{pmatrix}1&\bm{0}\\
^{t}\bm{0}&O(3)
\end{pmatrix}\right\}\subset O(1,3).$
By using Lemma \ref{lem:1}, we can relate coefficient matrices of frames of different types. The curvature of a frame of type D is related to the curvature of a frame of type F in a simple way as in \cite{NN}.

Now let us prove the following, which is a combination of Theorem \ref{thm:hier}-(2) and a part of Corollary \ref{thm:2reg}:
\begin{theorem}\label{th:DtoC}
If a regular time-like curve $\gamma$ admits a generalized Bishop frame of type D, then $\gamma$ admits a generalized Bishop frame of type C. In particular,
every time-like $2$-regular curve $\gamma$ admits a frame of type C.
\end{theorem}

We will need two lemmas to prove Theorem \ref{th:DtoC}.

\begin{lemma}\label{lem:2}
Let
\begin{equation}\label{eq:BC}
X_{\BB}= \begin{pmatrix}
0&b_{1}&b_{2}&b_{3}\\
b_{1}&0&0&0\\
b_{2}&0&0&0\\
b_{3}&0&0&0
\end{pmatrix}, \quad\quad 
X_{\CC} = \begin{pmatrix}
0&c_{1}&c_{2}&0\\
c_{1}&0&0&c_{3}\\
c_{2}&0&0&0\\
0&-c_{3}&0&0
\end{pmatrix}, 
\end{equation} 
where $b_1, b_2, b_3, c_1, c_2, c_3$ are functions on an interval $I$. Let $\gamma : I \to \LL^{4}$ be a regular time-like curve which has a Bishop frame $\{\TT, \BB_1, \BB_2, \BB_3\}$ whose coefficient matrix is $X_{\BB}$. 
Assume that $\TT' = d_1\DD_1$ for a smooth unit normal vector $\DD_1$ and a function $d_{1}$ and $\DD_1$ is nowhere tangent to $\BB_2$. Express $\DD_1$ as $\DD_1 = y_1 \BB_1 + y_2\BB_2 + y_3\BB_3$ for some functions $y_1, y_2, y_3$. 
Then $\gamma$ admits a generalized Bishop frame $\CC$ of type C whose coefficient matrix $X_{\CC}$, where $c_1,c_2,c_3$ are given by 
\begin{equation}\label{eq:cs}
c_{1}=\pm\sqrt{b_{1}^{2}+b_{3}^{2}}, \quad c_{2}=b_{2}, \quad c_{3}=\pm\frac{y_{3}^{\prime}y_{1}-y_{3}y_{1}^{\prime}}{y_{1}^{2}+y_{3}^{2}},
\end{equation}
and the transformation $G = \CC \BB^{-1}$ is of the form 
\begin{equation}\label{eq:bGbc1}
\left(\begin{matrix}
1&0&0&0\\
0&\pm\cos \theta &0&\pm\sin \theta\\
0&0&1&0\\
0&-\sin \theta&0&\cos \theta\\
\end{matrix}\right)
\end{equation}
for some function $\theta$ such that $\theta'=\pm \frac{y_{3}^{\prime}y_{1}-y_{3}y_{1}^{\prime}}{y_{1}^{2}+y_{3}^{2}}$. 
\end{lemma}
Note that the signs of entries of the coefficient matrix are different from the Euclidean case and we can use $O(3)$ to transform into a frame of type $\BB$ to a frame of type $\CC$ for time-like curves as in \cite{NN}.
The following lemma directly follows from Lemma \ref{lem:1} in the case where both $\ZZ_0$ and $\ZZ_1$ are frames of type B.

\begin{lemma}\label{lem:Bishop}
If a time-like curve $\gamma$ admits a Bishop frame $\BB$ whose coefficient matrix is $X_{\BB}$, then the coefficient matrix of any other Bishop frame is of the form 
\begin{equation}\label{eq:Qc}
\begin{pmatrix}1 & \o \\ {}^{t}\o & Q\end{pmatrix}
X_{\BB}\begin{pmatrix}1 & \o \\ {}^{t}\o & Q^{-1} \end{pmatrix}
\end{equation}
for some constant matrix $Q \in O(3)$. Conversely, for any constant $Q \in O(3)$, there exists a Bishop frame whose coefficient matrix is \eqref{eq:Qc}.
\end{lemma}
Lemma \ref{lem:Bishop} is almost the same as the Euclidean case, but the signs of the entries of the coefficient matrix are different.
\begin{proof}[Proof of Theorem \ref{th:DtoC}]
Assume that $\gamma$ admits a generalized Bishop frame of type D. Then, by Proposition \ref{proposition:1}, we have a unit normal vector $\DD_1$ and a smooth function $d_1$ such that $\TT'=d_1\DD_1$.  
Let $\DD_1 = y_1\BB_1 + y_2\BB_2 +y_3\BB_3$.
Consider $f : I \to \RR P^{2} ; s \mapsto [(y_1,y_2,y_3)(s)]$. By Sard's theorem, $f$ is not surjective. Take a unit vector $\bm{\xi} \in \RR^3$ such that $[\bm{\xi}] \in \RR P^{2} - \operatorname{Image} f$. Take $Q \in O(3)$ so that $\bm{\xi} Q^{-1} = (0,1,0)$. Then, by Lemma \ref{lem:Bishop}, $\gamma$ admits a Bishop frame $\{\TT, \hat{\BB}_1, \hat{\BB}_2, \hat{\BB}_3\}$ whose coefficient matrix is $\begin{pmatrix} 0 & {\bm b}Q^{-1}  \\ Q{}^{t}{\bm b} & O \end{pmatrix}$. Since $(y_1,y_2,y_3)Q^{-1}$ is never tangent to $(0,1,0)$, the vector $\DD_1$ is never tangent to $\hat{\BB}_2$. Then, by Lemma \ref{lem:2}, $\gamma$ admits a frame of type C.
\end{proof}
Note that the signs of entries of the coefficient matrix are different and we can use $O(3)$ to transform a frame of type $\BB$ to a frame of type $\BB$ for time-like curves as in \cite{NN}.
By Lemma \ref{lemma:0}, every regular time-like curve admits a Bishop frame. Therefore, by Corollary \ref{cor:2regD} and Theorem \ref{th:DtoC}, we get Corollary \ref{thm:2reg}.

The rest of hierarchy of frame is Theorem \ref{thm:noCD}-(3). It follows from the following proposition, which can be proved in the same way as in \cite{NN}.
\begin{proposition}\label{prop:typeC}
Let $\gamma$ be a regular time-like curve in $\LL^4$ with a Bishop frame $\BB$ whose coefficient matrix is $X_{\BB}=\begin{pmatrix}
0&b_{1}&b_{2}&b_{3}\\
b_{1}&0&0&0\\
b_{2}&0&0&0\\
b_{3}&0&0&0
\end{pmatrix}$, where $b_1, b_2, b_3$ are given by
\begin{align*}
b_{1}(s) & =\begin{cases}
e^{\frac{1}{s}} & s<0 \\
e^{-\frac{1}{s-2}} & 2<s\\
0 & \text{else},
\end{cases} \\
b_{2}(s) & =\begin{cases}
e^{-\frac{1}{s(-s+1)}} & 0<s<1 \\
0 & \text{else},
\end{cases} \\
b_{3}(s) & =\begin{cases}
e^{-\frac{1}{(s-1)(-s+2)}} & 1<s<2 \\
0 & \text{else}.
\end{cases}
\end{align*}
Then $\gamma$ does not admit a frame of type C.
\end{proposition}
\section*{Acknowledgments}
It is a pleasure to thank Professor Hiraku Nozawa for encouragement, useful remarks and suggestions. I would also like to acknowledge the support of Ritsumeikan University.

\end{document}